\documentclass{article}

\usepackage{arxiv}

\usepackage[utf8]{inputenc} 
\usepackage[T1]{fontenc}    
\usepackage{amsmath,amssymb,amsthm}
\usepackage{bm,bbm}
\usepackage{mathtools}
\usepackage{hyperref}       
\usepackage{url}            
\usepackage{amsfonts}       
\usepackage{microtype}      
\usepackage{graphicx}
\usepackage{doi}

\newtheorem{theorem}{Theorem}
\newtheorem*{theorem*}{Theorem}

\newtheorem{proposition}[theorem]{Proposition}
\newtheorem{corollary}[theorem]{Corollary}

\theoremstyle{remark}
\newtheorem{remark}[theorem]{Remark}

\newtheorem*{claim*}{\bf Assertion}

\title{The Holonomy of Optimal Mass Transport:\\ The Smooth Case}

\author{\href{https://orcid.org/0000-0000-0000-0000}{\includegraphics[scale=0.06]{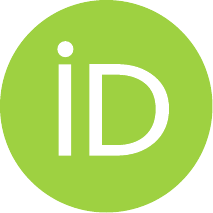}\hspace{1mm}Mahmoud  Abdelgalil} \\
	Mechanical and Aerospace Engineering\\
	University at Buffalo, SUNY\\
	Buffalo, NY 14260 \\
	\texttt{maabdelg@buffalo.edu} \\
	\And
	\href{https://orcid.org/0000-0000-0000-0000}{\includegraphics[scale=0.06]{orcid.pdf}\hspace{1mm}Tryphon T.~Georgiou} \\
	Mechanical and Aerospace Engineering\\
	University of California, Irvine\\
	Irvine, CA 92697\\
	\texttt{tryphon@uci.edu}\\
}

\date{}

\renewcommand{\headeright}{}
\renewcommand{\undertitle}{}
\renewcommand{\shorttitle}{}
\newcommand{\M}{M}
\newcommand{\Cinf}{C^{\infty}(\M)}
\newcommand{\cconv}{$\tfrac{1}{2}d^2$-convex}
\newcommand{\X}{\mathfrak{X}(\M)}
\newcommand{\XH}{\mathfrak{X}_H(\M)}
\newcommand{\Diff}{\mathrm{Diff}_0(\M)}

\DeclareMathOperator{\grad}{\mathrm{grad}}
\DeclareMathOperator{\expg}{\mathrm{exp}^{\text{$g$}}}
\DeclareMathOperator{\Div}{div}

\DeclareMathOperator{\Hess}{Hess}
\newcommand{\GOPT}{\mathcal{G}_{\mathrm{opt}}}
\newcommand{\FOPT}{\mathcal{F}_{\mathrm{opt}}}
\newcommand{\GPOT}{\mathcal{G}_{\mathrm{pot}}}
\newcommand{\FPOT}{\mathcal{F}_{\mathrm{pot}}}

\begin{document}
\maketitle

\begin{abstract}
We prove that, on a smooth manifold $\M$ equipped with a constant rank horizontal distribution $H$ and a smooth inner product $g_H$ on $H$, any horizontal vector field can be written as a linear combination of, at most, $3N$ Lie brackets of horizontal gradient fields of depth-$(1)$, where $N$ is the minimal immersion dimension of $M$, recovering the Riemannian case trivially when $H$ is the entire tangent bundle. When $H$ is also bracket-generating with a uniformly bounded step, we show that any vector field can be written as a linear combination of a uniformly bounded number of iterated Lie brackets of horizontal gradient fields with uniformly bounded depth, both bounds depending solely on the immersion dimension $N$ and the uniform upper bound on the step of $H$. Utilizing this, in conjunction with the Trotter property, we show that, if the manifold is compact, connected, and without boundary, the group generated by flows of horizontal gradient fields is dense in the identity component of the diffeomorphism group. When the manifold is also Riemannian, we show that the same density statement holds for the group generated by diffeomorphic optimal mass transport maps.
\end{abstract}

\keywords{Optimal Mass Transport, Gradient Vector fields, Group of Diffeomorphisms.}

\section{Introduction}
 
Given two smooth positive probability measures\footnote{A smooth positive probability measure is a probability measure that admits a smooth nowhere vanishing density.} $\nu_{1}$ and $\nu_{2}$
on a compact, connected, Riemannian manifold $\M$, the classical Optimal Mass Transport (OMT) problem seeks a map that minimizes the transport cost
\begin{align}\label{eq:mongecost}
    \varphi \;\longmapsto\; \frac{1}{2}\int_{\M}
    d\bigl( x,\, \varphi(x) \bigr)^{2}\, \mathrm{d}\nu_{1}(x),
\end{align}
where $d(\cdot,\cdot)$ denotes the Riemannian distance, among all measurable maps
satisfying the constraint
\begin{align}\label{eq:pushconstraint}
    \varphi_{\sharp}\nu_{1} = \nu_{2},
\end{align}
i.e., such that $\nu_{2}$ is the pushforward measure of $\nu_{1}$ under
$\varphi$. The minimizer of \eqref{eq:mongecost} subject to
\eqref{eq:pushconstraint} exists, and is $\nu_1$-almost everywhere unique and characterized by McCann's theorem \cite{McCann2001} as
\begin{align}\label{eq:mccannmap}
    \varphi_\star^{1\mapsto2} = \expg\bigl( \grad \phi_\star^{1\mapsto2}\bigr),
\end{align}
for a $\tfrac{1}{2}d^{2}$-convex\footnote{See the definition of $\tfrac{1}{2}d^{2}$-concave functions in \cite{McCann2001}, from which a function $\psi$ is defined as \cconv{} if $-\psi$ is $\tfrac{1}{2}d^{2}$-concave.} potential $\phi_\star^{1\mapsto2}$, where, for any $X\in\mathfrak{X}(\M)$, the map $\expg(X):x\mapsto\exp_x^g(X(x))$ is the Riemannian exponential map of $(\M,g)$.
The optimal map \eqref{eq:mccannmap} is the endpoint of the canonical interpolating family
\begin{align}\label{eq:displacement}
    \varphi^{1\mapsto2}_{s} := \expg\bigl( s \grad\phi^{1\mapsto2}_\star\bigr),
    \qquad s \in [0,1],
\end{align}
also known as McCann's
\emph{displacement interpolation} \cite{McCann1997}, and the induced curve of intermediate measures
$\nu_{s} := (\varphi^{1\mapsto2}_{s})_{\sharp}\nu_{1}$ turns out to be the unique constant-speed geodesic joining $\nu_{1}$ to $\nu_{2}$ in the $2$-Wasserstein space \cite{McCann2001,McCann1997}.
Under appropriate regularity hypotheses on the geometry of $\M$ \cite{MaTrudingerWang2005}, each measure $\nu_{s}$ admits a smooth positive density $\rho_s$, $\varphi_{s}^{1\mapsto2}$ is the optimal map from $\nu_{1}$ to $\nu_{s}$, and $\varphi^{1\mapsto2}_{s} \in \Diff$ for every
$s \in [0,1]$, where $\Diff$ denotes the \emph{identity component} of the diffeomorphism group of
$\M$. Moreover, the velocity field of \eqref{eq:displacement} is the gradient of a time-varying potential $\phi_{s}$ that solves the Hamilton--Jacobi equation
\begin{align}\label{eq:HJB}
    \partial_{s}\phi_{s} + \tfrac{1}{2}\|\grad\phi_{s}\|^{2} = 0, \qquad \phi_{0} = \phi^{1\mapsto2}_\star.
\end{align}

If we introduce another smooth positive probability measure $\nu_3$, then, in general,
$
\varphi^{1\mapsto3}_\star \neq \varphi^{2\mapsto3}_\star\circ\varphi^{1\mapsto2}_\star,
$
i.e., the class of solutions of the OMT problem is not closed under composition. A sharper formulation is obtained by considering the map
$
    \eta :=\varphi^{3\mapsto1}_\star \circ \varphi^{2\mapsto3}_\star\circ\varphi^{1\mapsto2}_\star,
$
which, under the same regularity hypotheses, is an element of $\Diff$, and, by construction, preserves $\nu_1$. In general, however, $\eta\neq \mathrm{id}$, i.e., traversing a closed triangle of optimal maps returns the measure to itself while rearranging the underlying manifold in a non-trivial way. In \cite{AbdelgalilGeorgiou2024,AbdelgalilGeorgiou2025,AbdelgalilGeorgiou2026}, we provide a precise characterization of the extent to which $\eta$ fails to be the identity in the finite-dimensional context of non-degenerate Gaussian distributions on $\mathbb{R}^n$. An equivalent problem was studied, without the OMT context, by Ballantine \cite{Ballantine1970}, who proved that every matrix of positive determinant is a product of at most five symmetric positive definite
factors, with the bound independent of the dimension.

Herein, our main motivation has been to investigate whether a
similar statement holds in the infinite-dimensional setting. Specifically, we shall prove the following statement (see Theorem \ref{thm:POT_OPT} below):
\begin{quote}
Suppose that $\M$ is a smooth, compact, connected Riemannian manifold without boundary. Then, every
$\varphi \in \Diff$ is the limit, in the $C^\infty$-topology, of finite compositions of OMT maps.
\end{quote}
As discussed above, the velocity field associated with diffeomorphic displacement interpolation is the gradient of a time-varying potential (cf.
\eqref{eq:displacement} and the Hamilton--Jacobi equation \eqref{eq:HJB} that follows it). Consequently, the reachable set of diffeomorphisms obtained by concatenating diffeomorphic displacement interpolations is related to the controllability properties of the infinite-dimensional nonholonomic right-invariant control system
\begin{align}\label{eq:controlsystem}
    \dot{\varphi} = \grad\phi_{t} \circ \varphi,
    \qquad \varphi(0) = \mathrm{id},
\end{align}
on $\Diff$, where $t\mapsto\phi_{t}$, a piece-wise continuous curve of smooth potentials, is the control input. In addition, similar to the finite dimensional case \cite{AbdelgalilGeorgiou2025}, controllability properties of \eqref{eq:controlsystem} pertain to the nature of the \emph{holonomy} group of Otto's principal bundle construction \cite{otto2001geometry}, \cite[Appendix A.5]{WendtKhesin2009}. As is well-known, however, the infinite-dimensional setting deprives us of the Chow--Rashevsky theorem. In addition, the space of gradient fields is not a module over smooth functions, so the powerful results in \cite{AgrachevCaponigro2009,ArguillereTrelat2017} do not apply to \eqref{eq:controlsystem}. 

Nevertheless, it turns out that \eqref{eq:controlsystem} is \emph{approximately controllable}, i.e., the reachable set of \eqref{eq:controlsystem} for time-varying gradient fields is dense in $\Diff$. In fact, a stronger statement holds: the group generated by flows of \emph{horizontal} gradient fields \cite{KhesinLee2009}, when the horizontal distribution is bracket-generating and has a uniformly bounded step, is dense in $\Diff$. In conjunction with the strong Trotter property \cite{Glockner2015}, this implies the same density statement for the group generated by diffeomorphic OMT maps.
Whether every element of $\Diff$ is a finite composition, rather than a limit, of diffeomorphic OMT maps, which is an open problem, is one possible generalization of Ballantine's problem \cite{Ballantine1970} to the nonlinear infinite-dimensional setting, the other being whether every element of $\Diff$ is a finite composition of flows of gradient fields. The two versions are equivalent in the Gaussian-Linear case which is thoroughly investigated by the authors in
\cite{AbdelgalilGeorgiou2024,AbdelgalilGeorgiou2025,AbdelgalilGeorgiou2026} and has been key in resolving a long-standing open question in linear systems theory \cite{AbdelgalilGeorgiou2026TAC}. A positive answer with uniform bounds on the number of factors would be key to an infinite-dimensional control problem posed in \cite{Brockett2007} and partially solved in \cite{AbdelgalilGeorgiou2026TAC}.

\subsection*{Note on bibliography}
The questions in this manuscript have been formulated, and intermittently pursued, by the authors since the finite-dimensional version was resolved \cite{AbdelgalilGeorgiou2024,AbdelgalilGeorgiou2025,AbdelgalilGeorgiou2026}. After completion of the manuscript, we became aware of the recent arXiv preprint by Kazandjian,
Mohsen, and Pozzoli \cite{KazandjianMohsenPozzoli2026}, which studies the Lie algebra generated by gradient fields.  Utilizing elliptic-operator methods and a theorem of Grabowski, the authors of \cite{KazandjianMohsenPozzoli2026} prove, for a compact manifold equipped with a non-degenerate symmetric
bilinear form, that \emph{iterated} Lie brackets of gradient fields of some finite, but \emph{unspecified}, depth-$(m)$ span $\X$, the space of smooth vector fields. 
Our results were obtained independently, are more elementary, have been motivated by different problems, and are sharper than \cite{KazandjianMohsenPozzoli2026} on multiple accounts. Specifically, our results do not require compactness of $\M$ and equally apply to manifolds with boundary. In addition, our results apply to the sub-Riemannian case with an explicit uniform upper bound on the depth and the number of brackets, and, in the Riemannian setting, we show that any vector field is in the linear span of, at most, $3N$ \emph{depth-$(1)$} Lie brackets of gradient fields, where $N$ is the minimal immersion dimension of $\M$. Finally, our results can be extended to pseudo-Riemannian manifolds, since the key identity \eqref{eq:absorption} holds for arbitrary symmetric non-degenerate bilinear forms.

\subsection*{Disclosure}
The absorption identity \eqref{eq:absorption}, key to our main result, was produced in interaction with Claude Fable 5 (Anthropic, 2026), and seems to be novel. The authors have verified the mathematical details and are solely responsible for all content.

\section{Preliminaries}

Unless otherwise stated, $\M$ is a smooth $n$-dimensional manifold equipped with a smooth constant rank distribution, i.e., a vector sub-bundle $H\subseteq T\M$, and a smooth inner product $g_H$ on $H$. When $H$ is bracket-generating, the triple $(\M,H,g_H)$ defines a sub-Riemannian manifold. For standard definitions and regularity notions for sub-Riemannian manifolds, e.g., the \emph{step} of $H$, we refer the reader to the comprehensive treatment \cite{AgrachevBarilariBoscain2020}. 
We write $\Cinf$ for smooth functions, $\X$ for smooth vector fields on $\M$, $\XH\subset\X$ for smooth \emph{horizontal} vector fields, i.e., sections of $H$, and $\grad_H\psi$ for the \emph{horizontal gradient} of $\psi\in\Cinf$ with respect to $g_H$, i.e., the unique horizontal vector field satisfying
\begin{align}
    d\psi(X) = g_H(\grad_H\psi,X),
\end{align}
for all $X\in\XH$. The \emph{depth} of an iterated Lie bracket is defined as the number of brackets in its expression, i.e., for any vector fields $X,Y\in \X$, $X$ is depth-$(0)$, $[X,Y]$ is depth-$(1)$, $[X,[X,Y]]$ is depth-$(2)$, and so on. Any smooth inner product $g_H$ on $H$ admits a (not necessarily unique) smooth extension to a Riemannian metric on $\M$. Fixing any such extension $g$, we use $\grad$, $\Hess$, $\nabla$, $\Div$, $\Delta$, and $\expg$ to denote the corresponding Riemannian gradient, Hessian, Levi-Civita connection, divergence, Laplace--Beltrami operator, and exponential map, respectively.
When $\M$ is compact, connected, and without boundary, the identity component of the group of smooth diffeomorphisms, denoted by $\Diff$, is a Fr\'echet Lie group with the standard $C^\infty$-topology \cite{Hamilton1982}. The exponential map of $\Diff$ is the map $X\mapsto \mathrm{e}^{X}$, for $X\in\mathfrak{X}(\M)$, where $\mathrm{e}^{X}$ is the solution at time $t=1$ of
\begin{align}
    \dot{\varphi} = X\circ \varphi, \qquad \varphi(0) = \mathrm{id},
\end{align}
which, when $\M$ is compact, exists, is unique, and is an element of $\Diff$. For a given $\varphi\in\Diff$, we use the notation $\varphi^{\circ k}$ to denote the $k$-times composition of $\varphi$ with itself, i.e., 
\begin{align*}
    \varphi^{\circ k}:= \underbrace{\varphi \circ \cdots \circ \varphi}_{k\text{-times}}.
\end{align*}
For a given family of diffeomorphisms $\mathcal{F}\subset\Diff$, we use $\langle \mathcal{F}\rangle$ to denote the subgroup generated by $\mathcal{F}$, i.e., 
\begin{align*}
    \langle \mathcal{F}\rangle :=\{\varphi_1\circ\varphi_2\circ\cdots\circ \varphi_k ~|~ k \in\mathbb{N}, \text{ and, for all  $i\leq k$, } \varphi_i\in \mathcal{F}, \text{ or } \varphi_i^{-1}\in \mathcal{F}\},
\end{align*}
equipped with the subspace topology, and we use $\overline{\mathcal{G}}$ to denote the closure of a subgroup $\mathcal{G}$ in $\Diff$. 

\section{Main Results}
Our first result identifies a crucial property of the real linear span of brackets of horizontal gradient fields.
\begin{proposition}\label{prop:bracket} Let $\Psi:\M\rightarrow\mathbb{R}^{N}$ be a smooth immersion. Then, every $X_H\in \XH$ admits the decomposition
$$
X_H = \sum_{i=1}^{3N} [\grad_H\phi_{2i-1}, \grad_H\phi_{2i}],
$$
for some collection of functions $\{\phi_i\}_{i=1}^{6N} \subset \Cinf$.
\end{proposition}

\begin{proof}
For any $\psi \in \Cinf$, let $\mathbb{D}(\psi)\in\mathfrak{X}(\M)$ denote the vector field defined by
\begin{align*}
    \mathbb{D}(\psi) := g_H(\grad_H \psi,\grad_H \psi)\, \grad_H\psi,
\end{align*}
and let $\{\psi_i\}_{i=1}^N$ be the component functions of the immersion $\Psi$. Define the endomorphism $P:H\rightarrow H$
\begin{align*}
    P:=\sum_{i=1}^{N} \mathbb{D}(\psi_i) \otimes \mathbb{D}(\psi_i)^{\flat_H},
\end{align*}
where $^{\flat_H}$ is the musical isomorphism induced by $g_H$. We claim that $P$ is smooth and admits a smooth inverse. To see this, note that, since $\Psi$ is an immersion, the differentials $\{\mathrm{d}\psi_i\}_{i=1}^{N}$ span the co-tangent bundle $T^{*}\M$ everywhere, and, therefore, the horizontal gradients $\{\grad_H\psi_i\}_{i=1}^{N}$ span the horizontal distribution $H$ everywhere. Recalling the definition of $\mathbb{D}(\psi)$ for a given $\psi\in\Cinf$, it is clear that the vector fields $\{\mathbb{D}(\psi_i)\}_{i=1}^{N}$ also span the horizontal distribution $H$ everywhere. Therefore, for
every $x\in\M$ and every $v\in H_x\backslash\{0\}$,
\begin{align*}
    g_H(x)(P(x)v,v) \;=\; \sum_{i=1}^{N} g_H(x)(\mathbb{D}(\psi_i)(x),v)^2
    \;>\; 0,
\end{align*}
so that $P$ is injective. Hence, $P$ is invertible and, by Cramer's rule, the inverse $P^{-1}$ is smooth. Consequently, every vector field $X_H \in \XH$ can be
decomposed as
\begin{align}\label{eq:wfields_decomp}
    X_H = PP^{-1}X_H=\sum_{i=1}^{N} \alpha_i\,\mathbb{D}(\psi_i),
    \qquad \alpha_i := g_H(\mathbb{D}(\psi_i),P^{-1}X_H),
\end{align}
where the coefficients $\alpha_i$ are smooth. We claim that each term in the summation \eqref{eq:wfields_decomp} is a linear combination of three Lie brackets of pairs of horizontal gradient fields with constant coefficients. Indeed, for all $\phi, \psi \in \Cinf$, we have that
\begin{equation}\label{eq:absorption}
\phi\, \mathbb{D}(\psi)
= -\tfrac{1}{4}\bigl[ \grad_H(\phi \psi^{2}),\, \grad_H \psi \bigr]
\;-\; \tfrac{1}{12}\bigl[ \grad_H \phi,\, \grad_H(\psi^{3}) \bigr]
\;-\; \tfrac{1}{4}\bigl[ \grad_H(\psi^{2}),\, \grad_H(\phi \psi) \bigr].
\end{equation}
To see this, note that, using the product rule for $\grad_H$ and elementary properties of the Lie Bracket, i.e.,
$$\grad_H(\phi \psi) = \phi\grad_H \psi + \psi \grad_H \phi, \qquad [X, \phi Y] = \phi [X, Y] + X(\phi)Y,$$ 
direct computation gives
\begin{align*}
\bigl[ \grad_H(\psi^{2}), \grad_H(\phi \psi) \bigr]
&= -2\psi^{2} [\grad_H \phi, \grad_H \psi] + 2\psi\,g_H(\grad_H \psi,\grad_H \psi) \grad_H \phi - 2\phi \,\mathbb{D}(\psi),\\
\bigl[ \grad_H \phi, \grad_H(\psi^{3}) \bigr]
&= 3\psi^{2} [\grad_H \phi, \grad_H \psi] + 6\psi\, g_H(\grad_H \phi,\grad_H \psi) \,\grad_H \psi,\\
\bigl[ \grad_H(\phi \psi^{2}), \grad_H \psi \bigr]
&= \psi^{2} [\grad_H \phi, \grad_H \psi] - 2\psi \,g_H(\grad_H \psi,\grad_H \psi)\grad_H \phi\\
&- 2\psi \, g_H(\grad_H \phi,\grad_H \psi)\grad_H \psi - 2\phi\, \mathbb{D}(\psi),
\end{align*}
which, after substituting into the right hand side of \eqref{eq:absorption} and simplifying, proves the identity.
Therefore, \eqref{eq:wfields_decomp} is a finite sum in which each term is, by virtue of \eqref{eq:absorption}, a linear combination of three Lie brackets of horizontal gradient fields with constant coefficients. Absorbing the coefficients into the functions, the conclusion follows.
\end{proof}

\begin{remark}
    We emphasize that Proposition \ref{prop:bracket} holds under fairly weak assumptions on $\M$ and $H$. In particular, it does not require compactness, and applies even when $\M$ has a non-empty boundary and $H\subsetneq T\M$. Moreover, it shows that the linear span, with constant coefficients only, of depth-$(1)$ Lie brackets of horizontal gradient fields is enough to express \emph{any} horizontal vector field \emph{without ever leaving the horizontal distribution}. Note that, in general, depth-$(1)$ Lie brackets of horizontal gradient fields may very well have non-horizontal components. When $H=T\M$, i.e., in the Riemannian case, $g_H=g$, $\grad_H=\grad$, and the proposition shows that any vector field can be decomposed into $3N$ depth-$(1)$ Lie brackets of pairs of gradient fields, where $N$ is the immersion dimension. The Whitney immersion theorem \cite[Theorem 6.18]{Lee2012} gives an upper bound of $N\leq 2n$ for manifolds with or without boundary. Sharper results, when applicable, may be invoked to supply the minimal immersion dimension. 
\end{remark}

\begin{remark}
When $H=T\M$, the content of Proposition~\ref{prop:bracket} is not merely that gradient fields are bracket generating, but that they are so \emph{exactly}, with finite sums, constant coefficients, and already at depth-$(1)$. Moreover, the required number of pairs of gradients is uniformly bounded in terms of the immersion dimension of $\M$ only. This result is remarkable since gradient fields are as far as possible from being a $\Cinf$-module. Nevertheless, Proposition \ref{prop:bracket} asserts that the module structure is recovered as soon as first-order Lie brackets are involved. It also asserts that the span of gradient fields is included in the span of Lie brackets of gradient fields. This is counterintuitive at a first glance, especially since it does not hold in the finite-dimensional counterpart \cite{AbdelgalilGeorgiou2025}. Nevertheless, it readily manifests if we compute the divergence of the bracket of two gradient fields. Indeed, for any $\psi_1,\psi_2\in\Cinf$, 
\begin{align*}
[\grad \psi_1,\grad \psi_2] = \Hess^\sharp \psi_2 (\grad \psi_1) - \Hess^\sharp\psi_1( \grad \psi_2), 
\end{align*}
where $\Hess^\sharp\psi$ is the unique $(1,1)$-tensor satisfying $g(\Hess^\sharp\psi (X),Y) = \Hess\psi(X,Y)$, for all $X,Y\in\X$ and $\psi\in\Cinf$. Computing the divergence of both sides then gives
\begin{align*}
    \Div([\grad \psi_1,\grad \psi_2]) = g(\grad \psi_1,\grad (\Delta \psi_2)) - g(\grad \psi_2,\grad (\Delta \psi_1)),
\end{align*}
which does not vanish generically, and, therefore, the Helmholtz decomposition of $[\grad \psi_1,\grad \psi_2]$ generally has a non-trivial gradient component. For example, if $\psi_1$ and $\psi_2$ are eigenfunctions of $\Delta$ with eigenvalues $\lambda_1$ and $\lambda_2$, then
\begin{align*}
    \Div([\grad \psi_1,\grad \psi_2]) = (\lambda_2-\lambda_1)g(\grad\psi_1,\grad\psi_2),
\end{align*}
which vanishes identically if and only if the eigenvalues coincide or the gradients are pointwise orthogonal.
\end{remark}

Our next result shows that, even if $H\subsetneq T\M$, every vector field can be written as a linear combination of a uniformly bounded number of iterated Lie brackets of horizontal gradient fields of a uniformly bounded depth, provided that $H$ is bracket-generating and has a uniformly bounded step.
\begin{proposition}\label{prop:bracket_full}
    Let $\Psi:\M\rightarrow\mathbb{R}^N$ be a smooth immersion, and $H$ be bracket-generating with step uniformly bounded by $r\geq 2$. Then, every $X\in \X$ can be written as a linear combination with constant coefficients of, at most,
    \begin{align*}
        3N\Big(1+\tbinom{N}{2}\,\tfrac{N^{\,r-1}-1}{N-1}\Big),
    \end{align*}
    iterated Lie brackets of horizontal gradient fields, each of, at most, depth-$(r)$.
\end{proposition}
\begin{proof}
    Let $\{\psi_i\}_{i=1}^N$ be the components of $\Psi$, and, with $V_i:=\grad_H\psi_i$, define the Lie bracket monomials $\mathcal{V}_1:=\{V_1,\dots,V_N\},$  $\mathcal{V}_2:=\{[V_i,V_j]~|~i<j\}$, and
    \begin{align*}
        \mathcal{V}_{i+1}:=\{[V_j,V]~|~ 1\leq j\leq N, \,V\in\mathcal{V}_{i}\},
    \end{align*}
    for $i\geq 2$, then observe that, since $\Psi$ is an immersion, the elements of $\mathcal{V}_1$ span $H_x$ at every $x\in\M$. Consequently, similar to the proof of Proposition \ref{prop:bracket}, $\XH=\mathrm{span}_{\Cinf}\,\mathcal{V}_1$. Through repeated application of the Jacobi identity and the Leibniz rule, every iterated Lie bracket of horizontal vector fields with, at most, depth-$(r-1)$ can be written as a $\Cinf$-linear combination of $\mathcal{V}_{\leq r}:=\cup_{\ell=1}^r\mathcal{V}_\ell$. Consider the sequence $H=:H^{(1)}\subseteq H^{(2)} \subseteq \cdots \subseteq T\M$, where
    \begin{align*}
        H^{(i)}_x:= \mathrm{span}_{\mathbb{R}}\{V(x)~|~ V\in\mathcal{V}_{\leq i}\},
    \end{align*}
    for all $x\in\M$, and observe that $H_x^{(i)}$ is the $i$-th member of the flag of $H$ at $x$ \cite[Definition 10.1]{AgrachevBarilariBoscain2020}. Since the step of $H$ is uniformly bounded by $r$, we have that $H^{(r)}=T\M$. We claim that any $X\in\X$ admits the global decomposition
    \begin{align}
        X = \sum_{V\in\mathcal{V}_{\leq r}}\alpha^X_V V, \label{eq:pseudo_inverse}
    \end{align}
    for some functions $\{\alpha^X_V\}\subset\Cinf$. To see this, we let
    $P := \sum_{V\in\mathcal{V}_{\leq r}} V\otimes V^\flat,$
    where $^\flat$ is the musical isomorphism corresponding to any (fixed) Riemannian extension $g$ of the inner product $g_H$. For any $x\in \M$ and $v\in T_x\M\backslash\{0\}$,
    \begin{align*}
        g(x)(v,P(x)v) = \sum_{V\in\mathcal{V}_{\leq r}} g(x)(v,V(x))^2,
    \end{align*}
    which is strictly positive. Consequently, $P$ is a smooth injective endomorphism of $T\M$ and, by Cramer's rule, has a smooth inverse. Therefore, every vector field $X\in\X$ admits the decomposition
    \begin{align*}
        X = PP^{-1}X = \sum_{V\in\mathcal{V}_{\leq r}} g(V,P^{-1}X) V
    \end{align*}
    which is \eqref{eq:pseudo_inverse} with $\alpha_V^X=g(V,P^{-1}X)$.
    Let $X\in\X$ and compute a decomposition of the form \eqref{eq:pseudo_inverse}, rewritten here with the summation split by layers for convenience
    \begin{align}
        X = \sum_{\tilde{V}\in\mathcal{V}_{1}}\alpha_{\tilde{V}}^X \,\tilde{V} + \sum_{\tilde{V}\in\mathcal{V}_{2}}\alpha_{\tilde{V}}^X \,\tilde{V} + \cdots + \sum_{\tilde{V}\in\mathcal{V}_{r}}\alpha_{\tilde{V}}^X \,\tilde{V}. \label{eq:pseudo_inverse_2}
    \end{align}
    For all $i\geq 2$, each monomial $\tilde{V}\in \mathcal{V}_{i}$ in the summands of the decomposition \eqref{eq:pseudo_inverse_2} is of the form $[V_j,V]$ with $V\in \mathcal{V}_{i-1}$. Therefore, the Leibniz rule and Proposition~\ref{prop:bracket} give that
    \begin{align*}
        \alpha_{[V_j,V]}^X \,[V_j,V] = d\alpha_{[V_j,V]}^X(V) \, V_j + [\alpha_{[V_j,V]}^X\,V_j, V] =  d\alpha_{[V_j,V]}^X(V) \,V_j + \sum_{\ell=1}^{3N}[[\grad_H\phi_{2\ell-1},\grad_H\phi_{2\ell}], V], 
    \end{align*}
    where the functions $\{\phi_{\ell}\}_{\ell=1}^{6N}$ are those produced by Proposition~\ref{prop:bracket} for the horizontal vector field $\alpha_{[V_j,V]}^X\,V_j\in\XH$. Now observe that the second term in this decomposition is a summation of $3N$ iterated Lie brackets of horizontal gradient fields of depth-$(i)$ (recall that $\mathcal{V}_{i}$ is, by construction, a family of iterated Lie brackets of horizontal gradient fields of depth-$(i-1)$). In addition, the first term is a horizontal vector field that can be lumped with the first summation in the decomposition. Hence, if $k_i$ denotes the number of iterated Lie brackets of horizontal gradient fields sufficient to express the $i$-th summation in the decomposition \eqref{eq:pseudo_inverse_2} as a linear combination with constant coefficients, then Proposition~\ref{prop:bracket} gives that $k_1=3N$ and $k_{i}= 3N|\mathcal{V}_{i}|$, for all $i\in\{2,\dots,r\}$, where $|\mathcal{V}_{i}|$ is the cardinality of $\mathcal{V}_{i}$. Moreover, the maximum depth is achieved at the last layer, i.e., depth-$(r)$. It is clear that $|\mathcal{V}_2|\leq\tbinom{N}{2}$, and $|\mathcal{V}_{i+1}|\leq N|\mathcal{V}_i|$ for all $i\geq 2$. Therefore, the total number of iterated Lie brackets is
    \begin{align*}
        \sum_{i=1}^r k_{i} \leq 3N \left(1+\binom{N}{2}\sum_{\ell=2}^{r} N^{\ell-2}\right) = 3N \left(1+\binom{N}{2}\frac{N^{r-1}-1}{N-1}\right).
    \end{align*}
\end{proof}
With Proposition~\ref{prop:bracket} and Proposition~\ref{prop:bracket_full} in hand, we now prove our first density result, i.e., that the group generated by the flows of horizontal gradient fields is dense in the identity component of the diffeomorphism group. 
\begin{theorem}\label{thm:POT_Diff}
    Let $\M$ be compact, connected, and without boundary, and let $H\subseteq T\M$ be bracket generating. Then, with $\FPOT:=\{\mathrm{e}^{\grad_H \psi}~|~\psi\in\Cinf\}$ and $\GPOT:=\langle\FPOT\rangle$, we have that $\overline{\GPOT} = \Diff$.
\end{theorem}
\begin{proof}
The closure of a subgroup in a topological group is a subgroup. Since $\Diff$ is closed in the full group of diffeomorphisms, and $\GPOT \subseteq \Diff$, we have that $\overline{\GPOT}$ is a closed
subgroup of $\Diff$. Define the family of vector fields
$$
\mathfrak{L} :=\left\{ X \in \X : \mathrm{e}^{tX} \in \overline{\GPOT}
\text{ for all } t \in\mathbb{R}\right\},
$$
which is closed under real scaling by definition, and clearly contains horizontal gradients, i.e.,
\begin{align*}
    \left\{ \grad_H\phi ~:~ \phi\in\Cinf\right\}\subset \mathfrak{L}.
\end{align*}
By the Trotter property \cite[pp. 7-8]{Glockner2015}, we have that 
$$
\mathrm{e}^{t(X + Y)} = \lim_{k\rightarrow\infty} \left(\mathrm{e}^{k^{-1}t X}\circ\mathrm{e}^{k^{-1}t Y}\right)^{\circ k},
$$ 
which implies that $\mathfrak{L}$ is closed under addition, and by the commutator property \cite[pp. 7-8]{Glockner2015}
$$
\mathrm{e}^{-t^2[X, Y]} = \mathrm{e}^{t^2[Y, X]} = \lim_{k\rightarrow\infty} \left( \mathrm{e}^{k^{-1}t X}\circ\mathrm{e}^{k^{-1}t Y}\circ
\mathrm{e}^{-k^{-1}t X}\circ\mathrm{e}^{-k^{-1}t Y} \right)^{\circ k^2},
$$ 
which implies that $\mathfrak{L}$ is closed under the Lie bracket. Hence, $\mathfrak{L}$ is a Lie sub-algebra of $\X$ containing horizontal
gradient fields. If $H=T\M$, then $\mathfrak{L} = \X$ by
Proposition~\ref{prop:bracket}. If $H\subsetneq T\M$, then $\mathfrak{L} = \X$ by Proposition~\ref{prop:bracket_full}, since bracket generation and compactness of $\M$ imply that the step of $H$ is uniformly bounded. Indeed, for every $i$, the set of points at which the step of $H$ is, at
most, $i$ is open, since it is characterized by the non-vanishing of a determinant of values of the monomials $\mathcal{V}_{\leq i}$ in the proof of Proposition \ref{prop:bracket_full}. Therefore, by bracket generation, these sets cover $\M$, and, by compactness, finitely many
suffice, so that the step of $H$ is uniformly bounded by some $r\geq 2$. From the definition of $\mathfrak{L}$, it follows that the group $\overline{\GPOT} $ contains the family of time-one flows of vector fields
$$
\mathrm{e}^{\X} := \{ \mathrm{e}^{X} : X \in \X \},
$$ 
and, therefore, $\langle\mathrm{e}^{\X}\rangle\subseteq\overline{\GPOT}$. Because $\Diff$ is a simple group under the assumptions of the theorem \cite{Thurston1974}, and because $\langle\mathrm{e}^{\X}\rangle$ is a non-trivial normal subgroup, $\langle\mathrm{e}^{\X}\rangle = \Diff$ and, hence, $\overline{\GPOT}=\Diff$. 
\end{proof}
As an immediate consequence of the above, we have the following result.
\begin{corollary}
    In the setting of Theorem \ref{thm:POT_Diff}, the right-invariant control system 
    \begin{align*}
        \dot{\varphi} = \grad_H\phi_t\circ\varphi,
    \end{align*}
    is approximately controllable on $\Diff$.
\end{corollary}

We now utilize Theorem \ref{thm:POT_Diff}, specialized to the Riemannian case, to prove that the group generated by diffeomorphic OMT maps is dense in $\Diff$. To that end, we have the following result. 
\begin{theorem}\label{thm:POT_OPT}
    Let $\M$ be a compact, connected Riemannian manifold without boundary, and let
    \begin{align*}
        \FOPT&:=\left\{\expg(\grad\phi) ~|~\phi\in\Cinf\text{ is \cconv{}}\right\}\cap\Diff.
    \end{align*}
    Then, with $\GOPT:=\langle\FOPT\rangle$, we have that $\overline{\GOPT} = \Diff$.
\end{theorem}
\begin{proof} 
    Let $\psi\in\Cinf$ and, for any $k\geq 1$, let $\psi_k:=k^{-1}\psi$.
    For every $s\in[-1,1]$ and every $\ell\geq 0$, we have that
    $$d_{C^{\ell}}\bigl( \expg(s \grad \psi_{k}),\, \mathrm{id} \bigr)
\;\leq\; C_{\ell,\psi}\, k^{-1},$$
    where $d_{C^{\ell}}$ denotes the standard $C^{\ell}$-distance between $C^\infty$ maps on $\M$ in some atlas, and $C_{\ell,\psi}$ is some constant. Hence, in the limit $k\rightarrow\infty$, the map $\expg(s \grad \psi_k)$ converges uniformly for all $s\in[-1,1]$ to the identity map in the $C^\infty$ topology on $C^\infty(\M,\M)$, the space of $C^\infty$ maps from $\M$ to itself. Because $\Diff$ is open in $C^\infty(\M,\M)$ \cite{KrieglMichor1997}, and because the curve $s\mapsto \expg(s \grad \psi_k)$ is a smooth curve in $C^\infty(\M,\M)$ passing through the identity, there exists $k_1$ such that, for all $k\geq k_1$, $s\mapsto \expg(s \grad \psi_k)$ is a smooth curve in $\Diff$. On the other hand, by \cite[Theorem 13.5]{Villani2009}, there exists $k_2\geq 1$ such that, for all $k\geq k_2$, the function $\psi_k$ is \cconv{}. Combining both arguments, we obtain that, for all $k\geq k^\star_\psi:=\max\{k_1,k_2\}$, and all $s\in[-1,1]$,
    \begin{align}
        \expg(s\grad \psi_k) \in \FOPT,
    \end{align}
    In particular, the curve $\gamma: [-(k^\star_\psi)^{-1},(k^\star_\psi)^{-1}]\ni s\mapsto \gamma(s):\expg(s \grad \psi)\in\Diff$ is a smooth curve taking values in $\FOPT$ such that $\gamma(0)=\mathrm{id}$ and $\gamma'(0) = \grad \psi$. By the strong Trotter property \cite[pp. 7-8]{Glockner2015}, the limit
    $$
    \lim_{k\rightarrow \infty} \gamma\left(\tfrac{1}{k}\right)^{\circ k}  = \mathrm{e}^{\grad \psi}
    $$ 
    holds in the $C^\infty$-topology on $\Diff$ and, since each $\gamma\left(\tfrac{1}{k}\right)$ belongs to $\FOPT$ for all $k>k^\star_\psi$, we obtain that $\overline{\GOPT}$ contains all time-one flows of gradient fields, i.e., the family $\FPOT$ from the statement of Theorem \ref{thm:POT_Diff} with $H=T\M$. Therefore, $\GPOT\subset \overline{\GOPT}$ and, since $\overline{\GOPT}$ is a closed subgroup, we have that $\overline{\GPOT}\subset \overline{\GOPT}$. The conclusion follows by invoking Theorem \ref{thm:POT_Diff}, specialized to the Riemannian case.
\end{proof}
\section{Concluding Remarks}
Our main contribution has been to show that, under natural assumptions, horizontal gradient fields generate all vector fields via finite linear combinations, with constant coefficients, of iterated Lie brackets of uniformly bounded depth, recovering the Riemannian case trivially. We then used standard arguments to show the density of the group generated by time-one flows of horizontal gradient fields in $\Diff$, as well as the density of the group generated by diffeomorphic OMT maps in the Riemannian case. Because OMT maps in the sub-Riemannian setting are only almost everywhere approximately differentiable in general \cite{figalli2010mass}, a similar density statement for compositions of sub-Riemannian OMT maps is beyond the reach of the present approach. Nevertheless, combined with the non-holonomic Moser theorem of \cite{KhesinLee2009}, our results imply that the holonomy group in Otto's principal bundle construction \cite{otto2001geometry},\cite[Appendix A.5]{WendtKhesin2009} is dense in the structure group, even in the non-holonomic setting. Finally, we remark that the exact decompositions and uniform bounds supplied by Proposition~\ref{prop:bracket} and Proposition~\ref{prop:bracket_full} hint that the two nonlinear infinite-dimensional versions of Ballantine's problem discussed in the introduction may have solutions in an open neighborhood of the identity.

\section*{Acknowledgments}
The authors are grateful to Boris Khesin for his valuable feedback and his suggestion to extend an earlier version of this manuscript to the sub-Riemannian setting.

\bibliographystyle{unsrt}
\bibliography{refs}

\end{document}